\documentclass[reqno]{amsart}
\usepackage{amsmath,amssymb,latexsym}
\usepackage[T1]{fontenc}
\usepackage{dsfont}
\usepackage{enumerate}

\usepackage{enumerate,xspace}

\usepackage{soul}

\usepackage{mdwlist}

\usepackage{xcolor}

\usepackage{nicefrac}

\usepackage{etoolbox} 
\usepackage{tikz-cd}

\usepackage[pdftex]{hyperref}

\theoremstyle{plain}

\newtheorem{theorem}{Theorem}[section]
\newtheorem{lemma}[theorem]{Lemma}
\newtheorem{proposition}[theorem]{Proposition}
\newtheorem{corollary}[theorem]{Corollary}
\newtheorem{fact}[theorem]{Fact}

\newcounter{proofcount}
\AtBeginEnvironment{proof}{\stepcounter{proofcount}}

\newtheorem*{claim*}{Claim}
\makeatletter                  
\@addtoreset{claim}{proofcount}
\makeatother                   

\newenvironment{claimproof*}[1][Proof of the Claim.] 
{%
	\proof[#1]%
	
}
{%
	\endproof%
}

\theoremstyle{definition}
\newtheorem{definition}[theorem]{Definition}
\newtheorem{remark}[theorem]{Remark}
\newtheorem{example}[theorem]{Example}

\newcommand{\acleq}{\textnormal{acl}^{\textnormal{eq}}}

\newcommand{\tp}{\textnormal{tp}}
\newcommand{\stp}{\textnormal{stp}}
\newcommand{\U}{\textnormal{U}}

\newcommand{\Aut}{\textnormal{Aut}}

\newcommand{\Stab}{\textnormal{Stab}}

\newtheorem*{akn}{Acknowledgments}

\def\Ind#1#2{#1\setbox0=\hbox{$#1x$}\kern\wd0\hbox to 0pt{\hss$#1\mid$\hss}
	\lower.9\ht0\hbox to 0pt{\hss$#1\smile$\hss}\kern\wd0}

\def\ind{\mathop{\mathpalette\Ind{}}}

\def\notind#1#2{#1\setbox0=\hbox{$#1x$}\kern\wd0
	\hbox to 0pt{\mathchardef\nn=12854\hss$#1\nn$\kern1.4\wd0\hss}
	\hbox to 0pt{\hss$#1\mid$\hss}\lower.9\ht0 \hbox to 
	0pt{\hss$#1\smile$\hss}\kern\wd0}

\title{A (possibly new) structure without the canonical base property}
\author{Michael Loesch}
\date{\today}
\address{ \, Abteilung f\"ur Mathematische Logik, Mathematisches Institut,
  Albert-Ludwigs-Universit\"at Freiburg, Ernst-Zermelo-Stra\ss e 1, D-79104
  Freiburg, Germany}
\thanks{Research supported by the Deutsche Forschungsgemeinschaft (DFG, 
German Research Foundation) - Project number 2100310201, 
part of the ANR-DFG program GeoMod}
\email{loesch@math.uni-freiburg.de}
\begin{document}

\begin{abstract}
In this short note, we introduce a generalization of the canonical base 
property, called transfer of internality on quotients. A structural 
study of groups definable in  
theories with this property yields as a consequence  
infinitely many new uncountably 
categorical additive covers of the complex numbers without the canonical 
base property. 
\end{abstract}

\maketitle

\section{Introduction}

Recall that a stable theory is one-based if the canonical base of 
every strong type $\stp(a/B)$ is algebraic over the realization $a$.
In particular, the tuple $a$ is independent from $B$ over the intersection 
$\acleq(a)\cap \acleq(B)$. This characterization of the geometry of 
forking was later captured by Pillay \cite{aP00}, see also Evans 
\cite{dE03}, in terms of the ample hierarchy.
In particular, one-basedness coincides with 
non-1-ampleness for stable theories. Similarly, non-2-ampleness agrees 
with CM-triviality, a 
notion 
introduced by Hrushovski in his \emph{ab initio} construction \cite{eH93}. 

Hrushovski and Pillay \cite{HP87}
showed that one-basedness has strong consequences for definable (or rather 
interpretable)
groups in stable theories. Among other results, they showed that one-based 
groups are central-by-finite. Later, Pillay proved that CM-trivial groups 
of finite Morley rank are nilpotent-by-finite. 
 
The canonical base property (in short, CBP) provides a different 
generalization of
one-basedness, by replacing the role of algebraicity with almost 
internality to a particular 
$\emptyset$-invariant family of types. The study of groups according 
to this relativization was first considered by  Kowalski and 
Pillay, who showed \cite{KP06} that definable groups in a stable theory 
with the CBP are now central-by-(almost internal). 
Blossier, Martin-Pizarro and Wagner \cite{BMPW12} adapted the ample 
hierarchy to this relative context, which they called \emph{tightness}, 
and proved that a 2-tight group of finite  
Lascar rank is nilpotent-by-(almost 
internal).

If two types
$\stp(b/A_1)$ and $\stp(b/A_2)$ in a stable theory are almost 
internal to a fixed family of types  of 
Lascar rank one, then so is  $\stp(b/\acleq(A_1)\cap \acleq(A_2))$, 
whenever $A_1$ is independent from 
$A_2$ over  $\acleq(A_1)\cap \acleq(A_2)$. The latter always holds for  
one-based stable theories. A remarkable 
result of Chatzidakis \cite{zC12}
(and of Moosa \cite{rM11} under the stronger assumption of the UCBP)
 states
that theories of finite rank with the CBP \emph{transfer 
	$\Sigma$-internality on quotients} with respect to a fixed (but 
	arbitrary) invariant 
	family $\Sigma$ of types of Lascar rank one, that is,  the type 
	$\tp(b/\acleq(A_1)\cap \acleq(A_2))$ 
is almost $\Sigma$-internal, whenever both $\stp(b/A_1)$ and $\stp(b/A_2)$ 
are. 
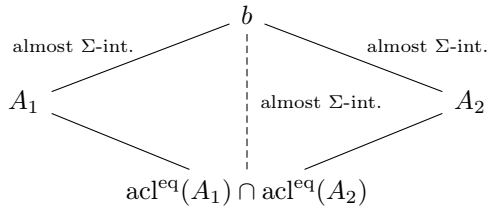
\begin{figure}[h]  
	\centering
		\begin{tikzcd}
			& b\arrow[dr, "\textnormal{ almost }\Sigma\textnormal{-int.}", dash]  
			\arrow[dl, 
			"\textnormal{ almost }\Sigma\textnormal{-int.}" above left, dash] 
			\arrow[dd, "\textnormal{ almost }\Sigma\textnormal{-int.}", dash, 
			dashed]
			&  \\ A_1\arrow[dr, dash] & &  
			A_2\arrow[dl, 
			dash] \\
			& \acleq(A_1)\cap \acleq(A_2) & 
		\end{tikzcd}
\end{figure} 

In this short note, we will undertake a first structural description of  
groups definable in stable theories of finite 
Lascar rank preserving $\Sigma$-internality on quotients.
As a consequence, we obtain infinitely many new stable (actually 
uncountably categorical)  structures 
 without the CBP, in terms of additive 
covers of the field of complex numbers, motivated by the primordial 
counter-example to the CBP given by Hrushovski, Palac\'in and Pillay 
\cite{HPP13}. 

\begin{akn}
	The author would like to thank his advisor Amador Mart\'in Pizarro for many helpful discussions, his advice, invaluable help and generosity.
\end{akn}

\section{Types, cosets and quotients}\label{S:CPIQ}

We fix a complete stable theory $T$ of finite Lascar rank and work inside a 
sufficiently saturated ambient model of $T$.
Let $\Sigma$ be an $\emptyset$-invariant family of types of finite 
rank. 
Motivated by work of Chatzidakis and of Moosa, we will study 
the following 
generalization of the canonical base property. 

\begin{definition}\label{D:PIQ}
	The theory $T$ \emph{transfers $\Sigma$-internality on quotients} if 
	for every imaginary element $b$ the type
	\[\stp(b/\acleq(A_1)\cap \acleq(A_2))\] 
	is almost $\Sigma$-internal, 
	whenever both $\stp(b/A_1)$ and $\stp(b/A_2)$ are almost 
	$\Sigma$-internal.
\end{definition}

\begin{remark}
	The above property is preserved under naming and forgetting parameters 
	\cite[Corollary 2.9]{mL20}.
	We will see before the Proposition \ref{P:CPIQM}
	that imaginary elements $b$ must be considered in the definition of 
	transfer of internality on quotients, 
	in contrast to other notions in geometric stability theory.	
\end{remark}
	
	Chatzidakis showed that the CBP already implies a seemingly stronger 
	property 
	(called UCBP), which appeared in work of Moosa and Pillay \cite{MP08} 
	on 
	the model theory of compact complex manifolds. 
	A key step in her argument consisted in showing that the CBP implies 
	transfer of internality on quotients (using our terminology).
	
\begin{fact}\label{F:Ch}\cite[Proposition 2.2 \& Lemma 2.3]{zC12}
	Every theory of finite Lascar rank with the CBP (with 
	respect to the family of types of rank one) 
	transfers internality on quotients with respect to any 
	$\emptyset$-invariant 
	family of types of rank one. 	
\end{fact}

As pointed out in the introduction, Chatzidakis' result (Fact \ref{F:Ch}) can be seen as a generalization of a general phenomenon for stable theories. 

\begin{remark}\label{R:IndepQuotients}
	 Stable theories transfer internality on \emph{independent quotients} with respect to any 
	$\emptyset$-invariant family $\Sigma'$ of types of rank one: 
	
	The type $\stp(b/\acleq(A_1)\cap\acleq(A_2))$ is almost 
	$\Sigma'$-internal, whenever
	the types $\stp(b/A_1)$ and 
	$\stp(b/A_2)$ are almost $\Sigma'$-internal and
	\[
	A_1 \ind_{\acleq(A_1)\cap\acleq(A_2)} A_2.
	\] In particular, one-based stable theories transfer $\Sigma'$-internality on quotients.
\end{remark}
\begin{proof}
	For simplicity of notation, we assume that $\acleq(A_1)\cap\acleq(A_2)=\emptyset$. By almost internality,
	there are sets of parameters $C_1$ and $C_2$ with
	\[  C_1 \ind_{A_1} b, A_2 \ \  \textnormal{ and  } \ \ C_2 \ind_{A_2} 
	b, C_1 \]
	such that $b$ is algebraic over $C_1,e_1$ and over $C_2,e_2$, where 
	$e_1$ and $e_2$ are tuples of realizations of types (each one based 
	over $C_i$) in $\Sigma'$. Note that $C_1 \ind C_2$. Choosing a maximal 
	subtuple $e_{1}'$ of $e_1$ which is independent from $C_2$ over $C_1$,
	we get that $e_1$ is algebraic over $C_1,e_{1}',C_2$ and $C_1,e_{1}' \ind C_2$.
	Since the type $\stp(e_1 / C_1,e_{1}')$ is almost $\Sigma'$-internal, we deduce from \cite[Lemma 1.7]{zC12}
	that $e_{1}$ is algebraic over $C_1,e_{1}',(C_2)_{\emptyset}^{\Sigma'}$, where $(C_2)_{\emptyset}^{\Sigma'}$ is the
	maximal almost $\Sigma'$-internal subset of $\acleq(C_2)$.
	Considering a maximal subtuple $e_{2}'$ of $e_2$ with
	\[ e_{2}' \ind_{C_2} C_1 , e_{1}',\]
	we similarly deduce that $e_{2}$ is algebraic over 
	$(C_1,e_{1}')_{\emptyset}^{\Sigma'}, 
	e_{2}',C_2$ and the 
	independence 
	\[ C_1 , e_{1} \ind_{(C_1 e_{1}')_{\emptyset}^{\Sigma'}, 
		(C_2)_{\emptyset}^{\Sigma'}} 
	C_2 , e_{2} \]
	follows. We conclude that $b$ is algebraic over $(C_1 
	e_{1}')_{\emptyset}^{\Sigma'}, 
	(C_2)_{\emptyset}^{\Sigma'}$, so the type $\stp(b)$ is 
	almost 
	$\Sigma'$-internal, as desired.
\end{proof}

We will now investigate structural properties of type-definable groups 
under transfer of  
$\Sigma$-internality 
on quotients, so
for the rest of this section fix a type-definable group $G$ over 
$\emptyset$. Given a type-definable 
subgroup $H$, the coset $aH$ can be seen as the hyperimaginary $[a]_{E}$, 
where $E(x,y)$ is the type-definable equivalence relation given by 
$x^{-1}y\in H$. 
Since stable theories eliminate hyperimaginaries, there is a (possibly 
infinite) tuple $\ulcorner aH\urcorner$ of imaginaries
which is interdefinable with the class $[a]_{E}$. This process is uniform with respect to $a$, 
since every type-definable group in a stable theory is an 
intersection of definable ones.

The generic elements of $G$ are those having maximal Lascar 
rank, which we refer to as the Lascar rank of $G$. For elements $a$ and 
$b$ 
independent over some set $C$, we will frequently use the following 
equivalence \[  b\cdot a \ind_C a \quad \Longleftrightarrow \quad 
\U( b\cdot a/C)=\U(b/C), \tag{$\clubsuit$}\]
 whilst $\U(b\cdot 
a/C)\geq\U(b/C)$ always holds, since $\U(b/C)=\U(b/C,a)$. In particular, 
if the element $a$ is generic over $b$, so is $a\cdot b$ generic over 
$b$. Ziegler 
noticed a sort of converse:

\begin{fact}\label{F:ZieglersLemma} \textup{\big(}\cite[Lemma 1.2]{BMPW16} 
	\textnormal{ \& } \cite[Theorem 1]{mZ06}\textup{\big)}
	If the elements $a,b$ and $a\cdot b$ are pairwise independent over 
	$C$, then each one is generic in an
	$\acleq(C)$-type-definable coset of the same connected type-definable subgroup 
	\[\Stab(a/C)=\Stab(b/C)=\Stab(a\cdot b/C).\]
\end{fact}

Motivated by Ziegler's lemma, we will introduce the following definition.

\begin{definition}\label{D:CF}
The 
	strong type $\stp(a/C)$ of an element $a$ of $G$ is \emph{coset-free} 
	if there is no proper
	type-definable connected 
	subgroup $H$ of the connected component $G^0$ such that $a$ is 
	contained in an
	$\acleq(C)$-type-definable 
	coset of $H$.
\end{definition}

Note that every generic type is coset-free. The following example of an 
additive
coset-free type in the field of complex numbers will
be used later.

\begin{example}\label{E:CosetFree}~
For a transcendental element $\alpha$ in the field $\mathbb C$ consider 
the type 
\[
p(x)=\stp(\alpha,\alpha^2,\ldots,\alpha^n).
\]
Every
$\mathbb{Q}^{\textnormal{alg}}$-definable coset of a definable additive 
subgroup is 
a linear variety over 
$\mathbb{Q}^{\textnormal{alg}}$, i.e. an 
intersections of subsets defined by
\[
\varepsilon_1 x_1 + \ldots \varepsilon_n x_n = \varepsilon_{n+1}
\]
for some algebraic elements $\varepsilon_i$.
Since $\alpha$ is transcendental, we deduce that the type $p(x)$ is 
coset-free in the additive group of the field of complex numbers.

\end{example}

The two following lemmata will be helpful for the study of type-definable 
groups in theories preserving internality on quotients.

\begin{lemma}\label{L:Lemma1}
	The product of $\U(G)$-many independent realizations of any coset-free 
	type $\stp(a/C)$ is generic over $C$. 
\end{lemma}
\begin{proof}
Let $(a_i)_{i\geq 1}$ be a Morley sequence of $\stp(a/C)$.
For $m\geq 1$, set $b_m = a_1\cdots a_m$. 
We will show that for $n=\U(G)$ the element $b_n$ is generic. 
Since $b_{m+1}=b_m\cdot a_{m+1}$,  we have the increasing chain 
\[ \U(b_1/C)\leq\U(b_2/C)\leq\cdots\leq\U (b_{n}/C)\leq n,\] as remarked 
after the equivalence $(\clubsuit)$, which we will implicitly use 
constantly in the following. Hence, there is an index $m\le n$ 
such that $\U(b_m/C)=\U(b_{m+1}/C)$. 
Therefore (by the equivalence $(\clubsuit)$), \[b_{m+1} \ind_C a_{m+1},\]
so
\[ a_2 \cdot a_3 \cdot \ldots \cdot a_{m+2} \ind_C a_{m+2},\]  by 
indiscernability of the Morley sequence. 
As 
\[ a_1 \ind_C  a_{m+2},a_2 \cdot a_3 \cdot \ldots \cdot a_{m+2}, \]
we deduce that
\[ b_{m+2} \ind_C a_{m+2},\] 
and hence, \[\U (b_{m+2}/C)=\U(b_{m+1}/C)=\U (b_{m}/C).\]
Iterating this process, it follows  that $\U(b_n/C)$ is maximal among the 
ranks $\U(b_j/C)$, which yields the following independences: 
\[b_n \ind_C a_{n+1} \cdot \ldots \cdot a_{2n} , \qquad
b_n \ind_C b_{2n} \quad \textnormal{ and } \quad
a_{n+1} \cdot \ldots \cdot a_{2n} \ind_C b_{2n}.\]
Fact \ref{F:ZieglersLemma} implies
that the coset $b_n \Stab(b_n/C)$ is type-definable over $\acleq(C)$. 

In order to conclude that $b_n$ is generic in $G$, we need only show that 
the connected type-definable subgroup
$H=\Stab(b_n/C)$ equals the connected component $G^0$.  Therefor we use 
that the type  $\stp(a/C)$ is 
coset-free.  Note that $a_n H= a_{n-1}^{-1}\ldots a_{1}^{-1}b_n H$. Since 
the sequence 
$(a_i)_{i\geq 1}$ is independent over $\acleq(C)$, we obtain the 
independence \[\ulcorner a_n 
H\urcorner\ind_C \ulcorner a_n 
H\urcorner,\] so the coset $a_n 
H$ is $\acleq(C)$-type-definable. We conclude that $H=G^0$, since 
$\stp(a/C)$ is coset-free. 
\end{proof}

\begin{lemma}\label{L:Lemma2}
	Suppose that the type $\stp(a/C)$ is coset-free and $g$ is generic 
	over $C,a$. The intersection $\acleq(g,C)\cap\acleq(g\cdot 
	a^{-1},C)=\acleq(C)$.
\end{lemma}
\begin{proof}
	For simplicity of notation, we assume that $C=\emptyset$.
	Let $(a_i)_{i\geq 1}$ be a Morley sequence of $\stp(a)$ independent 
	from $g$ with $a_1=a$.
	 Set now \[ g_1=g \text{ and } 
	g_{i+1}=g_i\cdot a_{2i-1}^{-1}\cdot a_{2i} \text{ for 
	} i\geq 1.\] 	
			By stationarity, the independence 
		\[a_{2i},a_{2i-1} \ind 
		g_i \cdot a_{2i-1}^{-1}\] 
		implies that
		\[
		a_{2i} \stackrel{\text{st}}{\equiv}_{g_i \cdot 
		a_{2i-1}^{-1}} 
		a_{2i-1},
		\]
		i.e. the elements $a_{2i}$ and $a_{2i-1}$ have the same strong 
		type over $g_i 
		\cdot a_{2i-1}^{-1}$.
	  Hence
		\[
		g_{i+1} \stackrel{\text{st}}{\equiv}_{g_i \cdot a_{2i-1}^{-1}} 
		g_{i},\] by construction, so\[
		\acleq(g_{i})\cap \acleq(g_i \cdot 
		a_{2i-1}^{-1})=\acleq(g_{i+1})\cap \acleq(g_i \cdot a_{2i-1}^{-1}).
		\]
		Again by stationarity $a_{2i}^{-1} 
		\stackrel{\text{st}}{\equiv}_{g_{i+1}} 
		a_{2i+1}^{-1}$ and thus 
		\[g_{i} \cdot a_{2i-1}^{-1} = g_{i+1} \cdot 
		a_{2i}^{-1} \stackrel{\text{st}}{\equiv}_{g_{i+1}} 
		g_{i+1} \cdot a_{2i+1}^{-1}.\]
		We deduce that the intersection
		\[
		\acleq(g_{i+1})\cap \acleq(g_i \cdot a_{2i-1}^{-1})=\acleq(g_{i+1})\cap \acleq(g_{i+1} \cdot a_{2i+1}^{-1}).
		\]
		Thus,
		\[
		\acleq(g_{i})\cap \acleq(g_i \cdot a_{2i-1}^{-1})=	\acleq(g_{i+1})\cap \acleq(g_{i+1} \cdot a_{2(i+1)-1}^{-1})
		\]
		for all $i\geq 1$. Hence, we obtain that
		\[
		\acleq(g)\cap\acleq(g\cdot 
		a^{-1})=\ldots=\acleq(g_n)\cap\acleq(g_{n}\cdot 
		a_{2n-1}^{-1})\subseteq\acleq(g)\cap\acleq(g_n).
		\]
		
		We need only show that the type 
		$\stp(a_1^{-1}\cdot a_{2})$ is also coset-free: Indeed, in that 
		case, the Lemma \ref{L:Lemma1} implies that the 
		element 
		\[b=(a_{1}^{-1} \cdot a_{2})\cdots
		(a_{2i-1}^{-1} \cdot a_{2i})\cdots
		(a_{2n-1}^{-1} \cdot a_{2n})
		\]
		is generic over $g$, where $n=\U(G)$. Thus, the product 
		$g_n=g\cdot b$ is generic and independent from $g$, so  the 
		intersection
		$\acleq(g)\cap\acleq(g_n)
		=\acleq(\emptyset)$, as desired. 
		 
		Suppose  that the element $a_1^{-1} \cdot a_{2}$ is 
		contained 
		in an 
		$\acleq(\emptyset)$-type-definable 
		coset of a connected subgroup $H$. The 
		independence 
		\[a_1, \ulcorner a_{1}^{-1}a_{2}U\urcorner \ind a_2\]
		implies that \[\ulcorner a_2 H\urcorner \ind \ulcorner a_2 
		H\urcorner,\] so
		$\ulcorner a_2 H\urcorner$ is $\acleq(\emptyset)$-definable. Since 
		$\stp(a)$ is coset-free, we deduce that $H=G^0$, as desired. 	
\end{proof}

We now have all the ingredients in order to show the following 
property of certain type-definable groups.

\begin{theorem}\label{T:CPIQ}
Suppose that the ambient theory $T$ transfers $\Sigma$-internality on 
quotients 
and consider a type-definable subgroup
$H$ of $G$ possibly over parameters. If the type $\stp(a)$ of the element 
$a$ of $G$ is coset-free and 
$\stp(\ulcorner aH \urcorner)$ is almost $\Sigma$-internal, then 
$\stp(\ulcorner gH \urcorner )$ is almost $\Sigma$-internal, 
whenever $g$ is generic in $G$ over $a$. 
\end{theorem}
Notice that we do not require that the type $\stp(a/\ulcorner H 
\urcorner)$ is coset-free. Indeed, if $\stp(a/\ulcorner H 
\urcorner)$ is coset-free, the same conclusion holds 
without imposing that the theory $T$ transfers 
$\Sigma$-internality on 
quotients: A finite product  of independent 
realizations of the type 
$\stp(a/\ulcorner H 
\urcorner)$ is 
generic in $G$, by the Lemma \ref{L:Lemma1}. 
\begin{proof}
	Let $g$ be generic in $G$ over $a$.
	The type $\stp(\ulcorner H \urcorner)$ is almost 
	$\Sigma$-internal, 
	because 
	$H=(aH)^{-1}aH$. Since the coset $gH$ is definable over $g,\ulcorner 
	H\urcorner$, the type \[\stp(\ulcorner gH 
	\urcorner/g)\]is almost $\Sigma$-internal.
	On the other hand, the identity $gH=(g\cdot a^{-1})aH$ yields that the 
	type 
	\[\stp(\ulcorner gH\urcorner/g\cdot a^{-1})\]
	is almost $\Sigma$-internal.  
	Preservation of $\Sigma$-internality on 
	quotients implies that 
	\[
	\stp(\ulcorner gH\urcorner/\acleq(g)\cap\acleq(g\cdot a^{-1}))
	\]
	is almost $\Sigma$-internal. Since the type $\stp(a)$ is coset-free 
	and $g$ is generic over $a$, the intersection
	\[
	\acleq(g)\cap\acleq(g\cdot a^{-1})=\acleq(\emptyset),
	\]
	by the Lemma \ref{L:Lemma2}, so the type $\stp(\ulcorner gH\urcorner)$ 
	is 
	almost $\Sigma$-internal, as desired.
\end{proof}
For a connected type-definable group $G$ in a theory with the CBP relative 
to the family $\Sigma$, Kowalski and 
Pillay \cite[Theorem 4.3]{KP06} showed that the quotient $G/Z(G)$ is 
almost $\Sigma$-internal. We do not know whether the same holds under the 
weaker assumption of transfer of $\Sigma$-internality on quotients, 
without assuming the existence of a suitable coset-free type $\stp(a)$, as 
in the previous theorem.
	
\section{A cover of a group}\label{S:Cover}

In this section, we will apply the previous results to covers of 
groups, which will then be needed in Section \ref{S:NoCBP} in order to 
produce new structures without the CBP. 

\begin{definition}\label{D:Cover}
	Suppose that $(G,\cdot)$ is a stable group of finite Lascar rank 
	(possibly 
	with additional structure). 
	A \emph{cover} of $G$ is a structure 
	$\mathcal{M}=(S,P,\pi,\star,\odot,\ldots)$ with two distinguished 
	sorts $S$ and $P$ such that the 
	following conditions hold:
	\begin{itemize}
		
			\item The sort $P$ equals $G$ and the induced structure on $P$ 
			coincides with the full structure of $G$.

		\item 	The map $\pi:(S,\odot)\rightarrow (P,\cdot)$ is 
	a surjective group homomorphism.
			\item The group action $\star$ of $P$ on $S$ turns each fiber 
	of $\pi$ into a principal homogeneous space.
	\end{itemize}
The cover is \emph{non-degenerated}, if the sort $S$ is not almost $P$-internal.
\end{definition}

Note in particular that $P$ is stably embedded in the cover $\mathcal M$, 
that is, every definable subset of $P^m$ is definable with parameters from $P$.
Hence, every type of a tuple in the cover over $P$ is definable over a 
subset of $P$. Furthermore, the cover $\mathcal M$ is again stable of 
finite rank, since each fiber is definably isomorphic to the sort $P$.

The example of an uncountably categorical theory without the CBP given by 
Hrushovski, Palac\'in and Pillay \cite{HPP13} may be seen as an (additive) 
cover of the complex numbers.

\begin{example}~\label{E:Cover}
	Consider the (additive) non-degenerated cover 
	$\mathcal{M}_1=(S,P,\pi,\star,\oplus,\otimes)$, where $P$ is the field 
	$\mathbb C$ of complex numbers and $S=\mathbb C \times 
	\mathbb C$. The projection $\pi:S\rightarrow P$ maps an element of $S$ 
	onto the first coordinate (seen as a pair of complex numbers) and 	
	the group action $\star$ of $P$ on $S$ is given by $\beta \star 
	(\alpha,a')=(\alpha, a' + \beta)$. 
	
	Moreover, the cover $\mathcal M_1$ has two distinguished 
	operations:   	
	\[	\begin{array}{rccc}
		\oplus: & S\times S & \rightarrow & S \\[2mm]
		& \big((\alpha,a'),(\beta,b')\big)&\mapsto& (\alpha+\beta, 
		a' + b');
		\end{array}\]
 and 
	\[	\begin{array}{rccc}
	\otimes: & S\times S & \rightarrow & S \\[2mm]
	& \big((\alpha,a'),(\beta,b')\big)&\mapsto& (\alpha\beta,\alpha 
	b'+\beta 
	a').
	\end{array}, \]
	which define a ring structure on $S$.
\end{example}

\bigskip
\textbf{Henceforth, we fix a non-degenerated cover 
$\boldsymbol{\mathcal{M}=(S,P,\pi,\star,\odot,\ldots)}$ of a 
group 
$\boldsymbol G$ of finite Lascar rank.}
\medskip

By $P$-internality or internality with respect to $P$ we mean internality with respect to the family $\Sigma$ of types in the sort $P$.

\begin{proposition}\label{P:CPIQG}
	Suppose that $\mathcal{M}$ transfers $P$-internality 
	on 
	quotients. 
	There exists no tuple $a=(a_1,\ldots,a_n)$ in 
	$S^n$ such that $\stp(\pi(a))$ is coset-free in $(P^n,\cdot)$ and 
	$a_n$ is 
	algebraic over $a_1,\ldots,a_{n-1},P$.
\end{proposition}

The idea of the proof is to mimic Theorem \ref{T:CPIQ}, despite we do not 
explicitly have a type-definable subgroup of $S^n$, whose coset should be 
the set $H_a$ in the proof below, due to the possible lack of 
compatibility between the group action $\star$ and the group law $\odot$.

\begin{proof}
	Since $P$ is stably embedded, there is a subset $B$ of $P$ such that 
	the type $p(x)=\stp(a/P)$ is definable over $B$.
	Choose a formula $\varphi(x, b)$ in $p(x)$  witnessing that $a_n$ is 
	algebraic over $a_1,\ldots,a_{n-1},B$. Set
	\[H=\big\{ \gamma\in P^n \ |\ \forall y\in P^{|b|} \big( 
	\textnormal{d}_p 
	x\varphi(x,y) \leftrightarrow \textnormal{d}_p x\varphi(\gamma^{-1}\star 
	x,y) \big) \big\}\]
	and
	\[
	H_{a}=\{ x\in S^n \ |\ x=\gamma\star a \textnormal{ for some } \gamma 
	\textnormal{ in } H\}.
	\]
	Note that the canonical parameter of the definable set $\ulcorner H_a 
	\urcorner$ is definable over $P$: Given an automorphism $\sigma$ 
	fixing $P$ pointwise, we have that $a$ and $\sigma(a)$ lie in the same 
	fiber with respect to the projection $\pi$, so $\sigma(a)=\gamma\star 
	a$ for some element $\gamma$ in $P^n$.   By the definition of $H$, it 
	is immediate that $\gamma$ belongs to $H$, so     
 $\sigma(a)$ lies in $H_a$. It  now 
	follows that $\sigma$ permutes $H_a$,  since the group $H$ is fixed pointwise. 
	
	Choose now a generic element $g=(g_1,\dots,g_n)$ in $S^n$ 
	over $a$. The set 
	\[g\odot 
	a^{-1} \odot H_a =  \{ g\odot 
	a^{-1} \odot x \ | \ x\in H_a \}
	\]
	 is definable over $g\odot a^{-1}, 
	\ulcorner H_a \urcorner$. Since every fiber of $\pi$ is $P$-internal, 
	so is the type
	\[
	\stp\big(\ulcorner g\odot a^{-1} \odot H_a \urcorner / \pi(g\odot 
	a^{-1})\big).
	\]
	We are now led to show that the type
	\[
	\stp(\ulcorner g\odot a^{-1} \odot H_a \urcorner)
	\]
	is almost $P$-internal.  By transfer of $P$-internality on 
	quotients,  it suffices to show that the type
	\[
	\stp\big(\ulcorner g\odot a^{-1} \odot H_a \urcorner / \pi(g)\big)
	\]
	is $P$-internal, for the intersection \[
	\acleq(\pi(g))\cap\acleq(\pi(g)\cdot\pi(a)^{-1})=\acleq(\emptyset),
	\] by Lemma \ref{L:Lemma2}. 
	
	Notice that  \[ \pi(a^{-1}\odot x) = \pi(a)^{-1}\cdot \pi(x) = e_G,\] 
	where $e_G$ is the neutral element of $G$, for $x$ in $H_a$. Thus, the 
	set 
	$g \odot a^{-1} \odot H_a$ is contained in the fiber of $g$, which 
	immediately yields that 
	the type
	\[
	\stp\big(\ulcorner g\odot a^{-1} \odot H_a \urcorner / \pi(g)\big)
	\]
	is $P$-internal, as desired. 
	
	In order to conclude, we need only show that the type 
	$\stp(g_n/g_1,\ldots,g_{n-1})$ is almost $P$-internal, since the 
	generic element $g$ of $S^n$ consists of generic independent 
	coordinates. 	
	Note that the set
	\[
	Z=\{ z\in S \ |\ (a_1,\ldots,a_{n-1},z)\in H_a \}
	\]
	is finite and contains $a$, witnessed by the formula $\varphi(x,b)$, 
	and 
	thus so is the set 
	\[
	\{ u \in S \ |\ (g_1,\ldots,g_{n-1},u)\in g\odot a^{-1}\odot H_a \} 
	=g_n\odot a_{n}^{-1} \odot Z. 
	\] Hence, the element 	$g_n$ is algebraic over 
	$g_1,\ldots,g_{n-1},\ulcorner 
	g\odot a^{-1} \odot H_a \urcorner$, so the type 
	$\stp(g_n/g_1,\ldots,g_{n-1})$ is almost $P$-internal, as desired.
	\end{proof}

We conclude this section showing that non-degenerated covers
of a group of Lascar rank one cannot
eliminate finite imaginaries, if the cover transfers 
$P$-internality on quotients. 
In particular, every such cover with elimination of imaginaries is 
a 
counterexample to the CBP.

\begin{corollary}\label{C:FiniteImagCovers}
	If the non-degenerated cover $\mathcal{M}$ of a group of Lascar rank one transfers $P$-internality 
	on quotients, then $\mathcal{M}$ does not eliminate finite imaginaries.
\end{corollary}
\begin{proof}
Choose two independent realizations $a$ and $b$ of the principal 
generic type of the group $(S,\odot)$ over $\acleq(\emptyset)$. 
We will show that the finite set $\{a,b\}$ has no real canonical parameter
$(c,\varepsilon)$, 
where $c$ is a tuple in 
$S$ and $\varepsilon$ is a tuple in $P$. 
\begin{claim*}
	The projection $\pi(c)$ is not algebraic over $\pi(a\odot b)$.
\end{claim*}
\begin{claimproof*}
	Since $a$ and $b$ are independent generic elements, their projections $\alpha=\pi(a)$ and $\beta=\pi(b)$ are distinct. Hence, the 
	 element $a$ is the unique element of the set $\{a,b\}$ in the fiber of $\alpha$. Thus, the generic $a$ of $S$ is definable over $c,P$.
	If the tuple $\gamma=\pi(c)$ were algebraic over $\pi(a\odot b)$, the independence $a\ind 
	a\odot b$ yields that the type $\stp(a)$ is almost $P$-internal, 
	contradicting our assumption that the sort $S$ is not almost 
	$P$-internal.
\end{claimproof*}
	Choose a coordinate $\gamma_i$ of $\gamma$ of rank one over $\pi(a\odot b)=\alpha\cdot\beta$.
	By Proposition \ref{P:CPIQG}, the type 
	$p=\stp(\alpha,\beta,\gamma_i)$ is not coset-free.
	Hence, the tuple $(\alpha,\beta,\gamma_i)$ is contained in an 
	$\acleq(\emptyset)$-definable coset of a proper connected type-definable subgroup $H$ 
	of $G^3$. We first show that $H$ equals the 
	stabilizer of $p$. Note that 
	\[
	2=\U(\alpha,\beta,\gamma_i)\leq\U\big((\alpha,\beta,\gamma_i)\cdot 
	H\big)=\U(H)<3,
	\]
	since $H$ is a proper subgroup of $G^3$.
	Thus, the type $p$ is the unique generic type 
	of the $\acleq(\emptyset)$-definable coset 
	$(\alpha,\beta,\gamma_i)\cdot H$. Hence, the stabilizer of $p$ contains $H$ and so they are equal, as desired.

 Now,  stationarity of the principal generic type of $S$ yields that 
 \[(a,b)\stackrel{\text{st}}{\equiv}(b,a).\] Choose 
 an automorphism $\sigma$ fixing $\acleq(\emptyset)$ with 
 $\sigma(a,b)=(b,a)$, so $\sigma(c)=c$. Since \[
(\alpha,\beta,\gamma_i)\cdot H = 
\big(\sigma(\alpha),\sigma(\beta),\sigma(\gamma_i)\big)\cdot 
\sigma(H)=(\beta,\alpha,\gamma_i)\cdot H,\]
 the tuple
$(\beta^{-1}\cdot\alpha,\alpha^{-1}\cdot\beta,1_G)$ is contained in $H$. 
The two elements $\alpha^{-1}\cdot \beta$ and $\beta$ realize the 
same type over $\acleq(\emptyset)$ (the principal generic type of $G$), so we deduce that 
$(\beta,\beta^{-1},1_G)$ belongs to $H$. Hence, the tuple
$(1_G,\alpha\cdot\beta,\gamma_i)$ is contained in the coset 
\[(1_G,\alpha\cdot\beta,\gamma_i)\cdot H=(\beta,\alpha,\gamma_i)\cdot H=(\alpha,\beta,\gamma_i)\cdot H.\] 
Since 
$\U(1_G,\alpha\cdot\beta,\gamma_i)=2$, the tuple $(1_G,\alpha\cdot\beta,\gamma_i)$ realizes $p$, the unique generic type 
of the coset $(\alpha,\beta,\gamma_i)\cdot H$, which is a blatant 
contradiction to $\alpha$ being generic.
\end{proof}
\section{Uncountably categorical covers without the CBP}\label{S:NoCBP}
In this section, we restrict our attention to additive covers of the field 
of complex numbers, in analogy to the Example \ref{E:Cover}, in order to 
produce uncountably categorical structures 
without the CBP, similar to the counter-example of
Hrushovski, Palac\'in and Pillay \cite{HPP13}.

\begin{definition}
	An \emph{additive cover} of the complex numbers $\mathbb C$ is a cover 
	of the additive group $(\mathbb C,+, 
	\cdot)$ (equipped with 
	the full field structure), where the sort $S$ equals $\mathbb C \times 
	\mathbb C$ and such that:
	\begin{itemize}
\item The projection $\pi:S\rightarrow P$ maps a pair onto the first 
coordinate.
\item The (lifted) group operation on $S$ is the map
\[	\begin{array}{rccc}
\oplus: & S\times S & \rightarrow & S \\[2mm]
& \big((\alpha,a'),(\beta,b')\big)&\mapsto& (\alpha+\beta, 
a' + b').
\end{array}\]
\item The group action $\star$ of $P$ on $S$ is given by 
$\beta\star(\alpha,a')=(\alpha,a'+\beta)$.
	\end{itemize}
\end{definition} 

Note that every additive cover of the complex numbers is an uncountably categorical structure with $P$ the unique strongly minimal set up to non-orthogonality.

\begin{remark}\label{R:Aut}
Given an additive cover $\mathcal{M}$, there is a 
canonical embedding
\[\begin{array}{rll}
\Aut(\mathcal{M}/P)& \hookrightarrow & 
\{F:\mathbb{C}\rightarrow\mathbb{C} 
\text{ 
	additive} 
\}\\[1mm]
\sigma &\mapsto & F_\sigma
\end{array}\]
uniquely determined by the identity
$\sigma(x)=F_\sigma(\pi(x))\star x$. 

For the additive cover $\mathcal M_1=(S,P,\pi,\star,\oplus,\otimes)$ 
in Example \ref{E:Cover}, the above embedding yields an isomorphism
\[\Aut(\mathcal{M}_{1}/P)\xrightarrow{\sim}\{D:\mathbb{C}\rightarrow\mathbb{C}
\text{ 
	derivation}\}.\]
As every derivation induces an automorphism fixing $P$ pointwise, the sort 
$S$ is not almost $P$-internal 
\cite[Corollary 3.3]{HPP13}, so the cover $\mathcal{M}_1$ is non-degenerated.

It was shown in \cite[Proposition 6.1]{mL20} that the cover $\mathcal M_1$ does not 
transfer almost $P$-internality on quotients.
We will now sketch an alternative argument, which 
will serve as a leitmotif in the sequel: Given a generic element $a_1$ in 
$S$, the product $a_2=a_1\otimes a_1$ is algebraic (or actually, 
definable) over $a_1$. The type $\stp(\pi(a_1),\pi(a_2))$ is coset-free 
in the additive group $P\times P$, by the Example \ref{E:CosetFree}. 
Therefore Proposition \ref{P:CPIQG} immediately gives that $\mathcal{M}_1$ 
does not transfer $P$-internality on 
quotients.
\end{remark}

In the previous short argument, the pair $(a_1, a_2)$ is the generic 
element of the a binary predicate $R_1$, namely the graph of the squaring 
function. With this in mind, we introduce the relations $R_n$.

\begin{definition}\label{D:Mn}
Let $\mathcal{M}_{n}$ be the additive cover of the complex numbers with 
the additional relation
\begin{multline*}
R_n\big((\alpha_1,a_{1}'),\ldots(\alpha_{n+1},a_{n+1}')\big) \iff \\ 
\bigwedge_{i=1}^{n+1} \alpha_i = \alpha_{1}^{i} 
\textnormal{ and } a_{n+1}'=\sum_{i=1}^{n}\binom{n+1}{i}(-1)^{n-i}\ \alpha_{n+1-i}\ a_{i}'
\end{multline*}
on $S^{n+1}$.
\end{definition}

A short calculation
yields that the group $\Aut(\mathcal{M}_n/P)$ corresponds to the collection of 
functions
\[
D_n = \big\{F:\mathbb C \rightarrow \mathbb C \textnormal{ additive}\ |\ 
F(\alpha^{n+1})=\sum_{i=1}^{n}\binom{n+1}{i}(-1)^{n-i}\ \alpha^{n+1-i}\ 
F(\alpha^i)\big\}.
\]

It is straightforward to see that $D_1$ is exactly the collection of 
derivations on $\mathbb C$. 
We will see that the $n^{\textnormal{th}}$ 
iterate of a derivation is contained in $D_n$, which explains why its 
elements are referred to as \emph{higher order derivations} \cite{HKR98}.
The set $D_n$ is the closure of the additive 
subgroup generated by the maps which are compositions of at 
most $n$-many derivations with respect to the product topology on the 
space of all maps $\mathbb C\rightarrow \mathbb C$, where the image 
$\mathbb C$ carries the discrete topology \cite[Theorem 
1.1]{KL18}. In particular, the collection $D_n$ is strictly contained in 
$D_{n+1}$. For completeness, we will now provide a self-contained proof of 
the latter, which is a verbatim adaptation of \cite[Proposition 4.5]{bE15} 
using the translation of \cite[Satz]{HKR98}.

\begin{proposition}\label{P:AutM}
	For every natural number $n\geq 1$, the collection $D_n$ is 
	contained in $D_{n+1}$. Given a non-trivial derivation $D$ 
	on $\mathbb 
	C$, the composition \[ D^{n+1}=\underbrace{D\circ\cdots\circ D}_{n+1 \textnormal{ times}} 
	\]
	belongs to $D_{n+1}\setminus D_{n}$. 
\end{proposition}
\begin{proof}
	We first show that an additive function $F$ is in $D_n$ if and only if
	\[
	F(x_1 \cdots x_{n+1})=\sum_{k=1}^{n}(-1)^{k+1}\sum_{1\leq 
		i_1 < \ldots < i_k \leq n+1} x_{i_1}\cdots x_{i_k} 
	F(x_1\cdots \widehat x_{i_1}\cdots \widehat x_{i_k}\cdots x_{n+1}),
	\]
	where $\widehat x_{j}$ indicates that the variable $x_j$ is omitted. 
	One direction is immediate, setting $x_1=\ldots=x_{n+1}=\alpha$. For 
	the other direction, let 
	$G_{F,n}(x_1, 
	\ldots, x_{n+1})$ denote the right-hand side of the 
	above equality.
	For $F$ in $D_n$, we have by definition the identity
	\begin{multline*}
	F((x_1 + \ldots +
	x_{n+1})^{n+1})=\\ \sum_{i=1}^{n}\binom{n+1}{i}(-1)^{n-i}\ (x_1 + 
	\ldots 
	+x_{n+1})^{n+1-i}\ F((x_1 + \ldots +x_{n+1})^i).
	\end{multline*}
	Since the function $F$ is additive, expanding the left-hand-side, it 
	is immediate to isolate that the only term which is odd in all the 
	variables $x_i$'s is exactly \[(n+1)!\cdot F(x_1 \cdots x_{n+1}).\] 
	Comparing to the terms on the right-hand-side which are also  odd in all 
	the variables, we deduce by renaming $n+1-i=k$ and dividing by 
	$(n+1)!$ that 
	\[ F(x_1 \cdots x_{n+1})=G_{F,n}(x_1,\ldots, x_{n+1}),\] 
	as desired. 
	
	Using the above equivalence, we now show that the function $F$ in 
	$D_n$  belongs to $D_{n+1}$.	Indeed,
	\begin{multline*}
	F(x_1 \cdots x_{n+2})=F(x_1 \cdots (x_{n+1}x_{n+2}))=G_{F,n}(x_1, 
	\ldots, x_{n+1}\cdot 
	x_{n+2})
	= \\
	G_{F,n+1}(x_1, \ldots, x_{n+1},x_{n+2})
	-x_{n+2}F(x_1 \cdots x_{n+1})-x_{n+1}F(x_1 \cdots \widehat x_{n+1} 
	x_{n+2})+\\
	x_{n+2}G_{F,n}(x_1, \ldots, x_{n+1})+
	x_{n+1}G_{F,n}(x_1, 
	\ldots, \widehat x_{n+1},x_{n+2})
	=\\ G_{F,n+1}(x_1, \ldots, x_{n+1},x_{n+2})
	-x_{n+2}\big(F(x_1 \cdots x_{n+1})-G_{F,n}(x_1, \ldots, x_{n+1})\big)
	- \\ x_{n+1}\big(F(x_1 \cdots \widehat x_{n+1} x_{n+2})-G_{F,n}(x_1, 
	\ldots, \widehat x_{n+1},x_{n+2})\big)=\\
	G_{F,n+1}(x_1, \ldots, x_{n+1},x_{n+2})
	-x_{n+2}\cdot 0
	- x_{n+1}\cdot 0 = G_{F,n+1}(x_1, \ldots, x_{n+1},x_{n+2}),
	\end{multline*}
	so $F$ is in $D_{n+1}$.
	
	Fix now some derivation $D$ on the complex numbers. We show 
	inductively on $n$ that the 
	composition $D^{n+1}$ belongs to $D_{n+1}$. Note that
	\[ D^{n+1}(x^{n+2})
	=D\big(D^{n}(x^{n+2})\big)
	=D\Big( \sum_{i=1}^{n+1}\binom{n+2}{i}(-1)^{n+1-i}\ x^{n+2-i}\ 
	D^{n}(x^i)  \Big), 
		\] since $D^n$ belongs to $D_n\subseteq D_{n+1}$. Hence, 
\begin{multline*}
 D^{n+1}(x^{n+2})=	\sum_{i=1}^{n+1}\binom{n+2}{i}(-1)^{n+1-i}\
	\big(  D(x^{n+2-i}) D^{n}(x^i) + x^{n+2-i} D^{n+1}(x^i) \big) = \\
	\sum_{i=1}^{n+1}\binom{n+2}{i}(-1)^{n+1-i}\
	x^{n+2-i} D^{n+1}(x^i) + 
	\sum_{i=1}^{n+1}\binom{n+2}{i}(-1)^{n+1-i}\
	D(x^{n+2-i}) D^{n}(x^i).
	\end{multline*}
We need only show that the subsum
	\[
	\sum_{i=1}^{n+1}\binom{n+2}{i}(-1)^{n+1-i}\
	D(x^{n+2-i}) D^{n}(x^i)=0.
	\]
	Indeed,
	\begin{multline*}
	\sum_{i=1}^{n+1}\binom{n+2}{i}(-1)^{n+1-i}\
	D(x^{n+2-i}) D^{n}(x^i)=\\
	D(x) \sum_{i=1}^{n+1}\binom{n+2}{i}(-1)^{n+1-i}\
	(n+2-i)x^{n+1-i} D^{n}(x^i)=\\
	D(x) \sum_{i=1}^{n+1}\binom{n+2}{i}(n+2-i)(-1)^{n+1-i}\
	x^{n+1-i} D^{n}(x^i)=\\
	D(x) \sum_{i=1}^{n+1}\binom{n+1}{i}(n+2)(-1)^{n+1-i}\
	x^{n+1-i} D^{n}(x^i)=\\
	D(x)(n+2) \Big(D^{n}(x^{n+1}) -\sum_{i=1}^{n}\binom{n+1}{i}(-1)^{n-i}\
	x^{n+1-i} D^{n}(x^i)\Big)=0,
	\end{multline*}
	since $D^{n}$ is in $D_n$.
	
	Assuming  now that the	composition $D^{n+1}$ belongs to the 
	collection $D_{n}$, we will deduce that $D$ is the trivial derivation. 
	We show inductively on $k\le n$ that $D^{n+1-k}$ must be contained 
	in $D_{n-k}$. Since $D_0=\{ \mathbf{0}\}$, this yields that $D$ 
	is trivial. The calculations will be very similar to the above.
	We have
	\begin{multline*}
	\sum_{i=1}^{n-k}\binom{n+1-k}{i}(-1)^{n-k-i}\
	x^{n+1-k-i} D^{n+1-k}(x^i) =
	D^{n+1-k}(x^{n+1-k}) =  \\  D\big(D^{n-k}(x^{n+1-k})\big) =
	D\Big( \sum_{i=1}^{n-k}\binom{n+1-k}{i}(-1)^{n-k-i}\
	x^{n+1-k-i} D^{n-k}(x^i)  \Big) = \\
	\sum_{i=1}^{n-k}\binom{n+1-k}{i}(-1)^{n-k-i}\
	D(x^{n+1-k-i}) D^{n-k}(x^i) + \\ + 
	\sum_{i=1}^{n-k}\binom{n+1-k}{i}(-1)^{n-k-i}\
	x^{n+1-k-i} D^{n+1-k}(x^i).
	\end{multline*}
Therefore,
	\begin{multline*}
	0= \sum_{i=1}^{n-k}\binom{n+1-k}{i}(-1)^{n-k-i}\
	D(x^{n+1-k-i}) D^{n-k}(x^i) = \\
= 	D(x)(n+1-k) \sum_{i=1}^{n-k}\binom{n-k}{i}(-1)^{n-k-i}\
	x^{n-k-i} D^{n-k}(x^i).	
	\end{multline*}
	We conclude that (even in case that $D(x)=0$) the sum
	\[ \sum_{i=1}^{n-k}\binom{n-k}{i}(-1)^{n-k-i}\
	x^{n-k-i} D^{n-k}(x^i)=0, \] for all $x$ in $\mathbb C$, so $D^{n-k}$ 
	belongs to $D^{n-(k+1)}$, as desired.
\end{proof}

In particular, the covers $\mathcal{M}_n$ have pairwise 
distinct automorphism groups.  However, at the moment of writing, we do 
not know whether these additive covers are bi-interpretable, and more 
importantly bi-interpretable with $\mathcal{M}_1$. 
A possible way to tackle this question is in terms of 
\emph{transfer of internality on intersections}, a property which is 
already present in Chatzidakis' proof of Fact \ref{F:Ch}. In 
\cite[Proposition 6.3]{mL20} it is shown that the additive cover 
$\mathcal{M}_1$ does not transfer $P$-internality on intersections.

Before showing that the additive covers $\mathcal{M}_n$, for $n\ge 1$, do 
not  transfer $P$-internality on quotients, notice that the presentation 
of $\mathcal M_1$ in the Example \ref{E:Cover} coincides with the 
Definition \ref{D:Cover}, in the sense that these two structures have the 
same  same definable 
sets. Indeed, the relation
$R_1(a_1,a_2)$ holds if and only if $a_2=a_1 \otimes a_1$, so the 
formula
\[
\psi(a,b,z)=\exists z_1 \exists z_2 \exists z_3 \Big(R_1(a,z_1) \land 
R_1(b,z_2) \land R_1(a \oplus b,z_3) \land 
z\doteq\frac{z_3 - z_2 - z_1}{2}\Big)
\]
defines $a\otimes b$.  Note that each cover $\mathcal{M}_n$ is a 
(definitional) reduct of $\mathcal{M}_1$:

\begin{multline*}
R_n(a_1,\ldots,a_{n+1}) \iff \\ 
\exists\varepsilon_2,\ldots,\exists\varepsilon_{n+1}\in P \Big(
\bigwedge_{i=2}^{n+1} a_i = \varepsilon_i \star ( \underbrace{a_1 \otimes \cdots\otimes a_1}_{i \textnormal{ times}} ) \  \land  \\
\varepsilon_{n+1}=
\sum_{i=2}^{n}\binom{n+1}{i}(-1)^{n-i}\ \pi(a_1)^{n+1-i}\ \varepsilon_i  \Big)
\end{multline*}

We will reproduce now the argument sketched  in 
the Remark \ref{R:Aut} in order to show that these new covers  do not 
transfer $P$-internality on quotients.  
Note first that there is no real tuple $b$ in $\mathcal M_n$ 
such that  the types $\stp(b/A_1)$ and $\stp(b/A_1)$ are 
both almost $P$-internal for some 
sets $A_1$ and $A_2$, but the type 
\[
\stp(b/\acleq(A_1)\cap\acleq(A_2))
\] is not almost $P$-internal. This is because the tuple
$\pi(b)$ has to be algebraic over $A_1$ and over $A_2$, 
by \cite[Corollary 3.3]{HPP13}, since every 
derivation induces an automorphism in $\Aut(\mathcal M_n/P)$.

\begin{proposition}\label{P:CPIQM}
	For each $n\geq 1$, the additive cover ${\mathcal M}_n$ does not transfer $P$-internality on 
quotients. In particular, the CBP does not hold in $\mathcal{M}_n$. 
\end{proposition}
Note that the so-called group version of the CBP \cite[Theorem 4.1]{KP06} yields a more direct argument for the failure of the CBP since
the stabilizer of a generic realization of $R_n$ is trivial.
\begin{proof}
Note that the sort $S$ is not almost $P$-internal, since $\mathcal M_n$ is a reduct 
of $\mathcal M_1$.
Choose a realization $a=(a_1,\ldots,a_n)$ of the relation $R_n$ with generic 
projection $\alpha=\pi(a_1)$. The 
element $a_n$ is definable over $a_1,\ldots,a_{n-1}, P$, by construction, 
since $D_n$ is (isomorphic to) $\Aut(\mathcal{M}_n/P)$. The type
$\stp(\pi(a))$ is coset-free by the Example \ref{E:CosetFree}. Hence, the 
result is a direct application of the Proposition \ref{P:CPIQG}. 
\end{proof}

Since the additive 
cover 
$\mathcal{M}_1$ eliminates finite imaginaries \cite[Corollary 4.5]{mL20},  Corollary \ref{C:FiniteImagCovers} yields another proof for the failure of the transfer of internality on quotients (and thus of the CBP) for $\mathcal{M}_1$. We will now conclude this work by showing that none of the additive covers $\mathcal{M}_n$, with $n>1$,  eliminate finite imaginaries. 
The main step consists in showing that subgroups of $G^n$ of the form \[ 
G(a/P)= \{ (g_1,\ldots, g_n) 
\in G^n \ | \  
(g_1\star a_1,\ldots,g_n\star a_n) \equiv_{P} (a_1,\ldots,a_n) 
\}\] for a suitable tuple $a=(a_1, \ldots, a_n)$ of $S^n$ will be 
$\acleq(\emptyset)$-definable. For additive covers, these definable groups 
are linear. Nonetheless, we will 
circumvent this description in the first part of the proof in order to 
provide a general criteria 
which can be easily adapted to arbitrary non-degenerated covers of a group of finite Morley rank. 

\begin{proposition}
	For $n>1$, the addive cover $\mathcal{M}_n$ does not eliminate finite 
	imaginaries.
\end{proposition}
\begin{proof}
	Assume for a contradiction that the finite set $\{a,b\}$ with $a$ and $b$ generic independent elements of $S$ has a real canonical parameter $(c,\varepsilon)$, 
	where $c=(c_i)_{i\leq m}$ is a tuple in 
	$S$ and $\varepsilon$ is a tuple in $P$.
	\begin{claim*}
	There is an index $j\leq m$ such that the group $G(a,b,c_j /P)$
	is not definable over $\acleq(\emptyset)$.
	\end{claim*}
	\begin{claimproof*}
		Suppose otherwise that all such groups are $\acleq(\emptyset)$-definable and
		choose an elementary substructure 
	$N$ of the cover $\mathcal{M}_n$ independent from $a,b$. We first 
	show 
	that for each $i\leq m$
		\[
		N\models \forall g_1   \forall g_2 \forall g_3 
		\big(
		(g_1,g_2,g_3)\in G(a,b,c_i /P)\rightarrow (g_2,g_1,g_3)\in 
		G(a,b,c_i 
	/P)
		\big).
		\]
		Given $(g_1,g_2,g_3)$ in $G(a,b,c_i /P)\cap N$ there is an 
	automorphism $\sigma$ in $\Aut(\mathcal M/P)$ with $\sigma(a)=g_1 
	\star 
	a,\sigma(b)=g_2 \star b$ and $\sigma(c)=g_3 \star c$. Note that 
	$(a,b) \equiv_N (b,a)$, so choose an automorphism $\tau$  
	fixing $N$ with $\tau(a,b)=(b,a)$. The tuple $c$ is fixed by $\tau$
	and the conjugate $\tau^{-1}\circ\sigma\circ\tau$ belongs to 
	$\Aut(\mathcal M /P)$. Since
		\[
		\tau^{-1}\circ\sigma\circ\tau (a)=g_2 \star a, \ \
		\tau^{-1}\circ\sigma\circ\tau (b)=g_1 \star b
		\ \ \textnormal{ and }\ \ 
		\tau^{-1}\circ\sigma\circ\tau (c)=g_3 \star c,
		\]
		the triple $(g_2,g_1,g_3)$ belongs to $G(a,b,c_i/P)$, as desired. 
		
		As $N$ was chosen to be an elementary substructure, we deduce that for each $i\leq m$
		\[
		\mathcal M \models \forall g_1   \forall g_2 \forall g_3 
		\big(
		(g_1,g_2,g_3)\in G(a,b,c_i /P)\rightarrow (g_2,g_1,g_3)\in 
		G(a,b,c_i 
	/P)
		\big).
		\]
		Now,  the sort $S$ is not $P$-internal, so $a$ is not definable over $b, P$. Hence, there is an automorphism 
	$\rho$ in $\Aut(\mathcal M /P)$ fixing $b$ with $\rho(a)=g\star a\neq a$. Set $\rho(c)=h \star c$.
	By the above, both tuples
		$(g,1_G,h)$ and $(1_G,g,h)$ belong to each $G(a,b,c_i /P)$, and hence so 
		does 	$(g,g^{-1},1_G)$. Therefore, there exists an automorphism $\rho_1$ 
		in $\Aut(\mathcal{M}/P)$ with $\rho_1(c)=c$ and 
		\[\rho_1(a)=g\star a.\]
	The element $g\star a$ does not lie in the fiber of $\pi(b)$ and differs from $a$, so $\rho_1(a)=g\star a$ is not contained in the set $\{a,b\}$, which gives the desired contradiction. 
		\end{claimproof*}
	In order to get the final contradiction, we will now use the description of the group $G(a,b,c_j /P)$ as a definable subgroup of $(\mathbb{C}^3,+)$. Every such 
	subgroup is given by a system of linear equations
	\[
	\varepsilon_1 x_1 + \varepsilon_2 x_2+ \varepsilon_3 x_3 = 0
	\]
	for some complex numbers $\varepsilon_i$. We will show that the 
	coefficients in each one of these linear equations can be taken in 
	$\mathbb{Q}^{\textnormal{alg}}$, which gives the desired 
	contradiction. Choose hence an arbitrary non-trivial linear equation 
	$\varepsilon_1 x_1 + \varepsilon_2 x_2+ \varepsilon_3 x_3 = 0$ which 
	occurs in the above system defining $G(a,b,c_j /P)$. Note that 
 $\varepsilon_3$ 
	cannot be zero, for otherwise (up to permutation of $a$ and $b$) every 
	automorphism fixing $a$ and $P$ would fix $b$, contradicting that the 
	cover $\mathcal{M}_n$ is non-degenerated. Therefore, we may normalize 
	and assume 
	that $\varepsilon_3=1$. 
	
	Set now $\alpha=\pi(a), \beta=\pi(b)$ and $\gamma=\pi(c_j)$. Choose 
	derivations $D_1$ and $D_2$ with $D_1(\alpha)=0\neq D_1(\beta)$ and 
	the kernel of $D_2$ is exactly $\mathbb{Q}^{\textnormal{alg}}$.
	As noted before Proposition \ref{P:AutM}, both additive maps $D_1$ and $D_2 \circ D_1$ are contained in the image
	of
	the canonical embedding 
	\[\Aut(\mathcal{M}_n/P)\hookrightarrow  
	\{F:\mathbb{C}\rightarrow\mathbb{C} 
	\text{ 
		additive} \},
	\]
	since $n>1$, so there are automorphisms $\sigma_1$ and $\tau$ in 
	$\Aut(\mathcal{M}_n/P)$ with
	\[
	\sigma_1(a)=D_1(\alpha) \star a=a \ \
	\sigma_1(b)=D_1(\beta) \star b \neq b
	\ \ \textnormal{ and }\ \ 
	\sigma_1(c_j)=D_1(\gamma) \star c
	\]
	and
\[
\tau(a)=D_2(D_1(\alpha)) \star a=a \ \
\tau(b)=D_2(D_1(\beta)) \star b
\ \ \textnormal{ and }\ \ 
\tau(c_j)=D_2(D_1(\gamma)) \star c.
\]
The above description of the group $G(a,b,c_j/P)$ yields
\[
	\varepsilon_1 D_1(\alpha) + \varepsilon_2 D_1(\beta)+ D_1(\gamma) =   
	\varepsilon_2 D_1(\beta)+ D_1(\gamma) = 0. \tag*{($\spadesuit$)}
\] 
Similarly,
\[
\varepsilon_1 D_2(D_1(\alpha)) + \varepsilon_2 D_2(D_1(\beta))+ 
D_2(D_1(\gamma)) =  \varepsilon_2 D_2(D_1(\beta))+ 
D_2(D_1(\gamma))= 0,
\] 
Differentiating ($\spadesuit$) with respect to $D_2$, we conclude from the 
above equation that 
\[
0=D_2(\varepsilon_2) D_1(\beta),
\]	
so $D_2(\varepsilon_2)=0$ and thus $\varepsilon_2$ lies in 
$\mathbb{Q}^{\textnormal{alg}}$.  Replacing now $D_1$ by a derivation 
$\tilde{D}_1$ with $\tilde{D}_1(\beta)=0\neq \tilde{D}_1(\alpha)$ we 
conclude that $\varepsilon_1$ is an algebraic number as well. 
\end{proof}

\end{document}